\documentclass[11pt]{amsart}
\usepackage{amssymb}
\usepackage[cyr]{aeguill}
\usepackage{cmap}
\usepackage{hyperxmp}
\usepackage[pdfdisplaydoctitle=true,
            colorlinks=true,
            urlcolor=blue,
            citecolor=blue,
            linkcolor=blue,
            pdfstartview=FitH,
            pdfpagemode= UseNone,
            bookmarksnumbered=true]{hyperref}

\newtheorem{theorem}{Theorem}
\newtheorem{corollary}[theorem]{Corollary}
\newtheorem{lemma}[theorem]{Lemma}

\theoremstyle{definition}
\newtheorem{definition}[theorem]{Definition}

\theoremstyle{remark}
\newtheorem{remark}{Remark}
\newtheorem{example}{Example}

\newcommand{\NN}{\mathbb{N}} 
\newcommand{\ZZ}{\mathbb{Z}} 
\newcommand{\RR}{\mathbb{R}} 
\newcommand{\CC}{\mathbb{C}} 

\newcommand{\Td}{\mathbb{T}^{d}} 
\newcommand{\Rd}{\mathbb{R}^{d}} 
\newcommand{\id}{\mathrm{id}} 

\newcommand{\Diff}{\mathrm{Diff}}%
\newcommand{\DiffRd}{\mathrm{Diff}_{H^{\infty}}(\mathbb{R}^{d})} 
\newcommand{\DiffTd}{\mathrm{Diff}^{\infty}(\mathbb{T}^{d})} 
\newcommand{\DRd}[1]{\mathcal{D}^{#1}(\mathbb{R}^{d})} 
\newcommand{\DTd}[1]{\mathcal{D}^{#1}(\mathbb{T}^{d})} 

\newcommand{\Vect}{\mathrm{Vect}} 
\newcommand{\HRd}[1]{H^{#1}(\mathbb{R}^{d},\mathbb{R}^{d})} 
\newcommand{\HTd}[1]{H^{#1}(\mathbb{T}^{d},\mathbb{R}^{d})} 
\newcommand{\HR}[1]{H^{#1}(\mathbb{R}^{d},\mathbb{R})} 
\newcommand{\HT}[1]{H^{#1}(\mathbb{T}^{d},\mathbb{R})} 
\newcommand{\CS}{\mathrm{C}^{\infty}} 

\newcommand{\norm}[1]{\left\Vert#1\right\Vert}
\newcommand{\abs}[1]{\left\vert#1\right\vert}
\newcommand{\set}[1]{\left\{#1\right\}}
\newcommand{\op}[1]{\mathrm{op}\left(#1\right)}

\DeclareMathOperator{\ad}{ad} %
\DeclareMathOperator{\Ad}{Ad} %
\DeclareMathOperator{\tr}{tr} %
\DeclareMathOperator{\Rec}{Rec} %
\DeclareMathOperator{\dive}{div} %

\hypersetup{
    pdftitle={Local well-posedness of the EPDiff equation: a survey}, 
    pdfauthor={Boris Kolev}, 
    pdfsubject={MSC 2010: 58D05, 35Q35}, 
    pdfkeywords={EPDiff equation; diffeomorphism groups; Sobolev metrics of fractional order}, 
    pdflang=EN 
    }

\begin{document}

\title[Local well-posedness of the EPDiff equation]{Local well-posedness of the EPDiff equation: a survey}

\author[B. Kolev]{Boris Kolev}
\address{Aix Marseille Universit\'{e}, CNRS, Centrale Marseille, I2M, Marseille, France}
\email{boris.kolev@math.cnrs.fr}


\date{\today}

\begin{abstract}
  This article is a survey on the local well-posedness problem for the general EPDiff equation. The main contribution concerns recent results on local existence of the geodesics on $\DiffTd$ and $\DiffRd$ when the inertia operator is a non-local Fourier multiplier.
\end{abstract}

\maketitle


\section{Introduction}
\label{sec:intro}

In a seminal paper~\cite{Arn1966}, published in 1966, Arnold recast the equations of motion of a perfect fluid (with fixed boundary) as the geodesic flow on the volume-preserving diffeomorphisms group of the domain (see also the short note of Moreau~\cite{Mor1959} going back to the late 1950s). For the little story, Arnold's paper was written in French for the bi-century of the 1765's paper of Euler~\cite{Eul1765a} (also written in French) who recast the equations of motion of a free rigid body as the geodesic flow on the rotation group.

This elegant geometrical re-formulation of ideal hydrodynamics applies, more generally, to any mechanical system when the configuration space has the structure of a Lie group $G$ and the Lagrangian is invariant by the lifted (right or left) action of $G$ on $TG$ (this idea goes back to Poincar\'{e}~\cite{Poi1901}). Consider, for instance, a Riemannian metric on $G$ which is right-invariant. Such a metric is thus specified by an inner product on $\mathfrak{g} = T_{e}G$, the Lie algebra of $G$. For historical reasons going back to the pioneering work of Euler~\cite{Eul1765a} on the motion of a rigid body, this inner product is usually represented by a linear isomorphism $A : \mathfrak{g} \to \mathfrak{g}^{*}$ called the \emph{inertia operator} and defined by
\begin{equation*}
  (Au)(v) := <u,v>, \qquad u,v \in \mathfrak{g}.
\end{equation*}
Let now $g(t)$ be a geodesic for this right-invariant metric on $G$ and consider the ``Eulerian velocity'', $u(t) := TR_{g^{-1}}.g_{t}$, where $g_{t}$ means time derivative, $R_{g}$ is the right translation on $G$ and $TR_{g}$ its tangent map. Then, the curve $u(t) \in \mathfrak{g}$ is a solution of the \emph{Euler--Arnold equation}
\begin{equation}\label{eq:Euler-Arnold}
  u_{t} = - \ad^{\top}_{u}u,
\end{equation}
first introduced\footnote{In Arnold's original paper, the derivation is done for a left-invariant metric and differs by a sign for a right-invariant metric (see also~\cite{AK1998,CK2012}).} in~\cite{Arn1966}, where $\ad_{u}v = [u,v]$ is the Lie bracket and $^{\top}$ means the adjoint relative to the inner product on $\mathfrak{g}$. This equation admits an integral form, also known as the \emph{conservation of the momentum} and which reads
\begin{equation}\label{eq:momentum-conservation}
  \Ad_{g(t)}^{\top}u(t) = u_{0},
\end{equation}
for any geodesic $g(t)$, where $\Ad_{g}^{\top}$ is the transpose of the adjoint action $\Ad_{g} := TL_{g}\circ TR_{g^{-1}}$ on $\mathfrak{g}$.

This elegant geometrical framework led, afterwards, to recast many partial differential equations relevant for mathematical physics as geodesic flows on various diffeomorphism groups: \emph{Burgers' equation} as the geodesic equation on $\Diff(S^1)$ with the $L^2$-metric, the \emph{KdV equation} on the Bott--Virasoro group with the $L^2$-metric in~\cite{KO1987,KM2003,CKKT2007}, the \emph{Camassa--Holm equation}~\cite{CH1993} on $\Diff(S^1)$ with the $H^1$-metric in~\cite{Kou1999, Mis1998,Mis2002}, the \emph{modified Camassa--Holm equation} on $\Diff(S^1)$ with the $H^k$-metric in~\cite{CK2003,MZ2009,HD2010}, the \emph{Hunter-Saxton equation} with the homogeneous $\dot H^{1}$-metric on $\Diff_{1}(S^1)$, the group of diffeomorphism of the circle which fix one point~\cite{Len2007,Len2008}, the \emph{modified Constantin--Lax--Majda equation}~\cite{CLM1985} as the geodesic equation on $\Diff_{1}(S^1)$ with respect to the homogeneous $\dot H^{1/2}$-metric~\cite{Wun2010,EKW2012,BKP2016}, the \emph{Euler--Weil--Peterson equation} as the geodesic equation on $\Diff_{3}(S^1)$, the group of diffeomorphism of the circle which fix three points, with respect to the homogeneous $\dot H^{3/2}$-metric~\cite{Gay2009}, the \emph{Degasperis--Procesi equation}~\cite{DP1999,DHH2002} as a non-metric Euler--Arnold equation~\cite{LMT2010,EK2011} on $\Diff(S^1)$.

From a geometrical view-point, this theory can be reduced to the study of a \emph{right-invariant Riemannian metric on the diffeomorphism group of a manifold $M$} (or one of its subgroup, or some extension of this group, or some right-invariant symmetric linear connection on this group).

To define a \emph{right invariant} Riemannian metric on the diffeomorphism group $\Diff(M)$ of a compact Riemannian manifold $M$, it suffices to prescribe an inner product on its Lie algebra $\Vect(M)$. We will moreover assume that this inner product can be written as
\begin{equation*}
  \langle u_{1},u_{2}\rangle := \int_{M} \left( Au_{1}\cdot u_{2} \right) d\mu \,,
\end{equation*}
where $u_{1}, u_{2} \in \Vect(M)$, $d\mu$ denotes the Riemannian density on $M$ and the \emph{inertia operator}
\begin{equation*}
  A : \Vect(M) \to \Vect(M)
\end{equation*}
is a $L^{2}$-symmetric, positive definite, continuous linear operator. By translating this inner product, we get an inner product on each tangent space $T_{\varphi}\Diff(M)$, which is given by
\begin{equation}\label{eq:definition-metric}
  G_{\varphi}(v_{1},v_{2}) = \int_{M}  \left( A_\varphi v_{1} \cdot v_{2} \right) \varphi^{*} d\mu \,,
\end{equation}
where $v_{1},v_{2}\in T_{\varphi}\Diff(M)$. Here $R_{\varphi} v := v \circ \varphi$ and
\begin{equation*}
  A_{\varphi} := R_{\varphi}\circ A\circ R_{\varphi^{-1}},
\end{equation*}
will be called the \emph{twisted map} (\textit{i.e.} the inertia operator $A$ twisted by the right translation $R_{\varphi}$).

A \emph{geodesic} for the metric $G$ is an extremal curve $\varphi(t)$ of the \emph{energy functional}
\begin{equation*}
  E(\varphi) := \frac{1}{2} \int_{0}^{1} G_{\varphi}(\varphi_{t},\varphi_{t}) \, dt,
\end{equation*}
where subscript $t$ in $\varphi_{t}$ means time derivative.

Let $u(t) := R_{\varphi^{-1}(t)}\varphi_{t}(t)$ be the \emph{Eulerian velocity} of the geodesic $\varphi(t)$. Then $u(t)$ is a solution of the \emph{Euler-Poincar\'{e} equation (EPDiff) on $\Diff(M)$}:
\begin{equation}\label{eq:EPDiff}
  m_{t} + \nabla_{u}m + \left(\nabla u\right)^{t}m + (\dive u) m  = 0, \quad m := Au \,,
\end{equation}
where $\left(\nabla u\right)^{t}$ is the Riemannian adjoint (for the metric on $M$) of $\nabla u$. In coordinates, using Einstein's summation convention, this equation reads
\begin{equation*}
  m_{t}^{i} + u^{j}\,m_{,j}^{i} + g^{il}\,g_{jk}\,u^{j}_{,l}\,m^{k} + u^{k}_{,k}\,m^{i} =0,
\end{equation*}
where $(g_{ij})$ is the Riemannian metric on $M$ and $(g^{ij})$, its inverse.

When $A$ is invertible, the EPDiff equation~\eqref{eq:EPDiff} can be rewritten as
\begin{equation}\label{eq:Diff-Euler-Arnold}
  u_{t} = - A^{-1}\left\{ \nabla_{u}Au + \left(\nabla u\right)^{t}Au + (\dive u) Au \right\},
\end{equation}
which is the \emph{Euler--Arnold equation} for $\Diff(M)$.

As acknowledged by Arnold himself, his seminal paper concentrated on the \emph{geometrical ideas} and not on the analytical difficulties that are inherent when \emph{infinite dimensional manifolds} are involved. In 1970, Ebin \& Marsden~\cite{EM1970} reconsidered this geometric approach from the analytical point of view (see also~\cite{EMF1972,Shn1985,Bre1989,Shn1994,Bre1999,Che2010}). They proposed to look at the \emph{Fr\'{e}chet Lie group} of smooth diffeomorphisms as an \emph{inverse limit of Hilbert manifolds}, following some ideas of Omori~\cite{Omo1970,Omo1997}. The remarkable observation is that, in this framework, the Euler equation (a PDE) can be recast as an ODE (the geodesic equation) on these Hilbert manifolds. Furthermore, following their approach, if we can prove \emph{local existence and uniqueness of the geodesics} (ODE), then the EPDiff equation~\eqref{eq:EPDiff} is \emph{well-posed}. They solved moreover the problem, when the inertia operator is a \emph{differential operator} (see also~\cite{Shk1998,Shk2000,CK2003,TY2005,Gay2009a,MP2010,MM2013}).

Most examples issued from mathematical physics correspond to integer $H^{k}$-metrics on diffeomorphism groups, for which the inertia operator is a differential operator. The first examples of Euler--Arnold equations with a \emph{non-local inertia operator} appear to be the modified Constantin--Lax--Majda equation~\cite{Wun2010,EKW2012} and the Euler--Weil--Peterson equation~\cite{Gay2009}

In this paper, we will consider the case when the ambient manifold $M$ is the torus $\Td$ or the Euclidean space $\Rd$, and when the inertia operator is a \emph{Fourier multiplier}. The Constantin--Lax--Majda equation and the Euler--Weil--Peterson equation are special occurrences of this theory, as well as every $H^{s}$ metric ($s$ real) on $\DiffTd$ or $\DiffRd$.

Classical arguments used to establish local existence of the geodesics when $A$ is a differential operator are no longer sufficient when $A$ is non-local and more work is required.

The goal of this survey is to present and summarize a series of studies~\cite{EKW2012,EK2014,BEK2015} on the local well-posedness problem for the general EPDiff equation on $\DiffTd$ or $\DiffRd$ when the inertia operator is a non-local Fourier multiplier.

In Section~\ref{sec:framework}, we recall basic material and fix notations. In Section~\ref{sec:reduction}, we show using two different methods (the \emph{spray method} and the \emph{particle-trajectory method}) that the local well-posedness problem reduces to show that the twisted map
\begin{equation*}
  \varphi \mapsto A_{\varphi} := R_{\varphi}\circ A\circ R_{\varphi^{-1}},
\end{equation*}
extends to a smooth map between some Hilbert approximation manifolds. Section~\ref{sec:twisted-map} is devoted to establish the smoothness of this extended twisted map, when $A$ is a Fourier multiplier of class $S^{r}$. Finally, in Section~\ref{sec:local-well-posedness}, we prove local well-posedness of the EPDiff equation, first in the Hilbert setting, and then in the smooth category, using \emph{a no loss, no gain} argument.


\section{Notations and background material}
\label{sec:framework}

In this paper, we consider first the group $\DiffTd$ of smooth diffeomorphisms of the $d$-dimensional torus isotopic to the identity. We want to consider this group as an ``infinite dimensional Lie group''. This requires first to define a differentiable structure on it and check that composition and inversion are smooth maps for this structure.

The first and most intuitive approach is to endow this group with a \emph{Fr\'{e}chet manifold} structure, modelled on the \emph{Fr\'{e}chet vector space} $C^{\infty}(\Td,\Rd)$, the space of $\ZZ^{d}$-periodic smooth maps from $\Rd$ to $\Rd$, with the topology defined by the family of semi-norms $\left(\norm{\cdot}_{C^{k}}\right)_{k \in \NN}$. Composition and inversion are smooth maps for this structure, and we can consider $\DiffTd$ as a \emph{Fr\'{e}chet-Lie group} as defined by Hamilton~\cite[Section 4.6]{Ham1982}. With this differentiable structure, the Lie algebra of $\DiffTd$ is $\Vect(\Td)$, the space of smooth vector fields on the torus, which is isomorphic to $C^{\infty}(\Td,\Rd)$. The Lie bracket is given by
\begin{equation*}
  [u,w] = du.w - dw.u.
\end{equation*}
Since moreover $\Td$ is compact, $\DiffTd$ is a \emph{regular Fr\'{e}chet Lie group} in the sense of Milnor~\cite{Mil1984}. In particular, each element $u$ of the Lie algebra $\Vect(\Td)$, corresponds to a one-parameter subgroup of $\DiffTd$.

There is however a serious weakness of the Fr\'{e}chet category. The topological dual of a Fr\'{e}chet space $E$ and more generally the space $\mathcal{L}(E,F)$ of continuous linear mappings between two Fr\'{e}chet spaces $E,F$ is not a Fr\'{e}chet space (unless $E,F$ are Banach spaces)~\cite{Ham1982}. This is annoying if one aims to extend differential geometry to \emph{Fr\'{e}chet manifolds}~\cite{AS1968,Kel1974}.

The definition of a ``good differentiable structure'' on a set $X$, in general, and on the diffeomorphisms group $\Diff(M)$ in particular, is a subtle topic and has been an active research area for decades. Many definitions have emerged, usually not equivalent~\cite{Omo1970,Mic1980,Ham1982,Mil1984,FK1988,Omo1997,KM1997,Igl2013}.

The most general framework is probably the category of \emph{diffeological spaces}~\cite{Igl2013}. These spaces are defined by a \emph{diffeology} (like a topological space is defined by a \emph{topology}) and differentiable mappings are defined as the morphisms of this structure (like continuous mappings between topological spaces). A less general framework, but suitable for calculus on \emph{manifolds of mappings}, is the \emph{convenient calculus} formalized by Krigel and Michor~\cite{Mic1980,KM1997}. It relies on \emph{Fr\"{o}licher spaces}~\cite{FK1988} which are themselves a subcategory of diffeological spaces (see~\cite[pages 99 and 390--391]{Igl2013}).

Coming back to $\DiffTd$, it has a stronger structure than just a \emph{Fr\'{e}chet Lie group}. Indeed, it is the inverse limit of Hilbert manifolds which are themselves topological groups. For this reason, it is called an \emph{ILH-Lie group} following Omori~\cite{Omo1997}. It will be the object of this section to describe these approximation manifolds.

We will also be interested in the diffeomorphism group of $\Rd$. But, since difficulties arise due to the non-compactness of $\Rd$, we cannot use the full group of smooth diffeomorphisms but need to restrict our study to some subgroup with nice behaviour at infinity. Several choices are possible, the subgroup of diffeomorphisms with compact support, the subgroup of rapidly decreasing diffeomorphisms, \ldots. However, the most suitable subgroup on which the theory works well is the following one
\begin{equation}\label{eq:Diff_infinity}
  \DiffRd := \set{\id + u;\; u \in \HRd{\infty} \; \text{and} \; \det(\id + du)> 0}\,,
\end{equation}
where $ \HRd{\infty}$ denotes the space of $\Rd$-valued $H^{\infty}$-functions on $\Rd$, i.e.,
\begin{equation*}
  \HRd{\infty} := \bigcap_{q \ge 0} \HRd{q}\,,
\end{equation*}
and where $\HRd{q}$ denotes the ($\Rd$-valued) Sobolev space on $\Rd$, defined below.

Let $\mathcal{F}$ be the Fourier transform on $\RR^{d}$, defined with the following normalization
\begin{equation*}
  \hat{f}(\xi) = (\mathcal{F} f)(\xi) = \int_{\Rd} e^{-2i\pi \langle x,\xi \rangle} f(x) \, dx
\end{equation*}
so that its inverse $\mathcal{F}^{-1}$ is given by:
\begin{equation*}
  (\mathcal{F}^{-1} \hat{f})(x) = \int_{\Rd} e^{2i\pi \langle x,\xi \rangle} \hat{f}(\xi) \, d\xi \, .
\end{equation*}
For $q\in \RR^{+}$ the Sobolev $H^{q}$-norm of a function $f$ on $\Rd$ is
defined by
\begin{equation*}
  \norm{f}_{H^{q}}^{2} := \norm{(1+\abs{\xi}^{2})^{\frac{q}{2}} \hat{f}}_{L^{2}}^{2}\, .
\end{equation*}
The Sobolev spaces $\HR{q}$ is defined as the closure of the space of compactly supported functions, $C_{c}^{\infty}(\Rd,\RR)$, relatively to this norm and the space $\HRd{q}$ is the space of $\Rd$-valued functions of which each component belongs to $\HR{q}$.

\begin{remark}
  In the case of the torus, we define similarly the space $\HT{q}$, as the closure of $C^{\infty}(\Td,\RR)$ for the norm
  \begin{equation*}
    \norm{f}_{H^{q}}^{2} := \sum_{\xi \in \ZZ^{d}} (1+\abs{\xi}^{2})^{\frac{q}{2}} \abs{\hat{f}(\xi)}^{2}\, ,
  \end{equation*}
  where
  \begin{equation*}
    \hat{f}(\xi) := \int_{\Td} e^{-2i\pi \langle x,\xi \rangle} f(x) \, dx \, .
  \end{equation*}
\end{remark}

It is worth to recall the following \emph{Sobolev embedding lemma} which proof can be found in~\cite[Proposition 2.2]{IKT2013}.

\begin{lemma}\label{lem:sobolev-embedding}
  Let $q > d/2$ then the space $\HRd{q+r}$ can be embedded into the space $C^{r}_{0}(\Rd,\Rd)$ of all $C^r$-functions vanishing at infinity and the space $\HTd{q+r}$ can be embedded into the space $C^r(\Td,\Rd)$, for any integer $r$.
\end{lemma}

We will also recall the following result on extension of pointwise multiplication to a bounded bilinear mapping between Sobolev spaces (see for instance ~\cite[Lemma 2.3]{IKT2013}).

\begin{lemma}\label{lem:pointwise-multiplication}
  Let $q > d/2$ and $0 \le p \le q$ then pointwise multiplication extends to a bounded bilinear mapping
  \begin{equation*}
    \HR{q} \times \HR{p} \to \HR{p}.
  \end{equation*}
  More precisely, there exists $C >0$ such that
  \begin{equation*}
    \norm{fg}_{H^{p}} \le C \norm{f}_{H^{q}} \norm{g}_{H^{p}},
  \end{equation*}
  for all $f \in \HR{q}$ and $g\in \HR{p}$. In particular $\HR{q}$ is a multiplicative algebra if $q > d/2$.
\end{lemma}

As already stated, $\DiffRd$ and $\DiffTd$ are \emph{ILH-Lie groups} (see Omori~\cite{Omo1997} for a precise definition). In simple words,
\begin{equation*}
  \DiffRd = \bigcap_{q > 1 + d/2} \DRd{q}, \qquad \DiffTd = \bigcap_{q > 1 + d/2} \DTd{q}\,,
\end{equation*}
where the sets $\DRd{q}$ and $\DTd{q}$ are defined, for $q>\frac{d}2+1$, as follows.
\begin{equation*}
  \DRd{q} := \set{\id + u;\; u \in \HRd{q} \; \text{and} \; \det(\id + du)> 0}\,,
\end{equation*}
and $\DTd{q}$ is the set of $C^{1}$ diffeomorphisms $\varphi$ of the torus $\Td$ isotopic to the identity and such that
\begin{equation*}
  \tilde{\varphi} - \id \in \HTd{q}\,,
\end{equation*}
where $\tilde{\varphi}$ is any lift of $\varphi$ to $\Rd$. Both of these sets are smooth Hilbert manifolds, modelled respectively on $\HRd{q}$ and $\HTd{q}$. For a more detailed treatment of these manifolds, we refer to~\cite{Ebi1970,IKT2013}.

\begin{remark}
  Note, that the tangent bundle $T\DRd{q}$ is a trivial bundle
  \begin{equation*}
    T\DRd{q}\cong\DRd{q}\times \HRd{q}\,,
  \end{equation*}
  because $\DRd{q}$ is an open subset of the Hilbert space $\HRd{q}$. The tangent bundle of the Hilbert manifold $\DTd{q}$ is also trivial but for a different reason. Indeed, let $\mathfrak{t}: T\Td \to \Td \times \Rd$ be a smooth trivialisation of the tangent bundle of the torus. Then
  \begin{equation*}
    \Psi : T\DTd{q} \to \DTd{q}\times \HTd{q}, \qquad \xi \mapsto \mathfrak{t} \circ \xi
  \end{equation*}
  defines a smooth vector bundle isomorphism because $\mathfrak{t}$ is smooth (see~\cite[page~107]{EM1970}).
\end{remark}

The Hilbert manifolds $\DTd{q}$ and $\DRd{q}$ are topological groups (see~\cite{IKT2013} for a modern exposition on the subject). For $\DTd{q}$, this is known since the 1960s~\cite{El1967,Ebi1968,Pal1968,Omo1970,EMF1972}. They are however not Hilbert Lie groups, because composition and inversion are continuous but not smooth (see ~\cite[Proposition 2.6]{IKT2013}). We recall also the following result concerning the right action of $\DRd{q}$ on $\HRd{p}$.

\begin{lemma}[Lemma 2.7 in~\cite{IKT2013}]\label{lem:composition}
  Given any two real numbers $q,p$ with $q>1 + d/2$ and $q\geq p\geq 0$, the mapping
  \begin{equation*}
    \mu^{p}: \HRd{p} \times\DRd{q} \rightarrow \HRd{p}, \qquad (u,\varphi) \mapsto u \circ \varphi
  \end{equation*}
  is continuous. Moreover, the mapping
  \begin{equation*}
    R_{\varphi}: u \mapsto u \circ \varphi
  \end{equation*}
  is \emph{locally bounded}. More precisely, given $C_{1}, C_{2}>0$, there exists a constant $C=C(p,C_{1},C_{2})$ such that
  \begin{equation*}
    \norm{R_{\varphi}}_{\mathcal{L}(H^{p},H^{p})} \leq C ,
  \end{equation*}
  for all $\varphi\in \DRd{q}$ with
  \begin{equation*}
    \norm{\varphi-\id}_{H^{q}} < C_{1} \quad \text{and} \quad \underset{x\in\Rd}{\inf}\left( \det(d\varphi(x))\right) > C_{2}.
  \end{equation*}
\end{lemma}

Finally, let $J_{\varphi}$ denotes the Jacobian determinant of a diffeomorphism $\varphi$ in $\DRd{q}$. From lemma~\ref{lem:pointwise-multiplication}, we deduce that the mapping
\begin{equation*}
  \varphi \mapsto J_{\varphi}, \qquad \DRd{q} \to \HR{q-1}
\end{equation*}
is smooth and we have moreover the following result, which is a reformulation of~\cite[Lemma 2.5]{IKT2013}.

\begin{lemma}\label{lem:inverse-Jacobian-smoothness}
  Let $q > 1 + d/2$ and $0 \le p \le q$. Given $\varphi \in \DRd{q}$ and $f \in \HR{p}$, the function $f/J_{\varphi}$ belongs to $\HR{p}$ and the mapping
  \begin{equation*}
    (\varphi,f) \mapsto \frac{f}{J_{\varphi}}, \qquad \DRd{q} \times \HR{p} \to \HR{p}
  \end{equation*}
  is smooth. The same result holds for the torus $\Td$.
\end{lemma}


\section{Reduction of the problem}
\label{sec:reduction}

It is well known that analysis in Fr\'{e}chet manifolds, like $\DiffTd$ or $\DiffRd$, is quite involved as the inverse function theorem, the implicit function theorem and the Cauchy--Lipschitz theorem do not hold~\cite{Ham1982,KM1997}. The strategy of Ebin \& Marsden~\cite{EM1970} to solve the well-posedness problem for the Euler equation was first to recast the equation as an ODE on some approximating Hilbert manifolds and prove local existence (in the Hilbert setting) by standard ODE techniques. Then, using symmetries of the equation, they proved that if the initial data is smooth, the solution remains smooth at each time (no loss, nor gain in spatial regularity) leading to a well-posedness result in the smooth category.

In this section, we will expose two ways to solve equation~\eqref{eq:Diff-Euler-Arnold}; the first one uses the second-order \emph{spray method} of Ebin-Marsden~\cite{EM1970}, while the other one is based on a reduction due to Ebin~\cite{Ebi1984} (see also Majda-Bertozzi~\cite{MB2002}, where the latter is described as the \emph{particle-trajectory method}). In both cases, we need to show (in order to apply standard ODE techniques) that a certain vector field defined, \textit{a priori}, in the smooth category extends smoothly to some approximating Hilbert manifolds.

We will introduce first these vector fields and show then that their extension (to some Hilbert manifolds) and smoothness reduce to prove that the \emph{twisted map}
\begin{equation*}
  A_{\varphi} := R_{\varphi}\circ A\circ R_{\varphi^{-1}},
\end{equation*}
\textit{i.e.} the inertia operator $A$ twisted by the right translation $R_{\varphi}$, extends smoothly to these approximating Hilbert manifolds.

\subsection{The spray method}
\label{subsec:spray}

Starting from the Euler--Arnold equation~\eqref{eq:Diff-Euler-Arnold}, we shall retrieve the Lagrangian formalism. Let $u(t)$ be a solution of the EPDiff equation and let $\varphi(t)$ be the flow of the time dependant vector field $u(t)$, \textit{i.e.}, $\varphi_{t} = u \circ \varphi$. Set $v := \varphi_{t}$ so that $v = u \circ \varphi$.

First, we want to compute the time derivative of the curve $v(t)$ in the bundle $T\Diff(M)$. To do so, we fix $x \in M$ and consider the vector field $X(t) = u(t,\varphi(t,x))$ on $TM$ defined along the curve $x(t) = \varphi(t,x)$ on $M$. Using the covariant derivative (in $M$) along the curve $x(t)$, we get in a local chart on $M$:
\begin{equation*}
  \left(\frac{D}{Dt}X(t)\right)^{k} = X^{k}_{t}(t) + \Gamma_{ij}^{k}(x(t))\,X^{i}(t)\,x^{j}_{t}(t).
\end{equation*}
But
\begin{equation*}
  X^{k}_{t}(t) = u^{k}_{t}(t,\varphi(t,x)) + \frac{\partial u^{k}}{\partial x^{l}}(t,\varphi(t,x))\,\varphi^{l}_{t}(t,x)
\end{equation*}
and
\begin{equation*}
  x^{j}_{t}(t) = \varphi^{j}_{t}(t,x) = u^{j}(t,\varphi(t,x)).
\end{equation*}
Hence
\begin{multline*}
  \left(\frac{D}{Dt}X(t)\right)^{k} =  u^{k}_{t}(t,\varphi(t,x)) + \frac{\partial u^{k}}{\partial x^{l}}(t,\varphi(t,x))\,u^{l}(t,\varphi(t,x))
  \\
  +  \Gamma_{ij}^{k}(\varphi(t,x))\,u^{i}(t,\varphi(t,x))\, u^{j}(t,\varphi(t,x))
\end{multline*}
and thus
\begin{equation*}
  v_{t} = u_{t} \circ \varphi + (\nabla_{u}u)\circ \varphi.
\end{equation*}

Therefore, we get
\begin{align*}
  v_{t} & = \left\{- A^{-1}\Big(\nabla_{u}Au + (\nabla u)^{t}Au + (\dive u) Au \Big) + \nabla_{u}u \right\}\circ \varphi \\
        & = \left\{ A^{-1}\Big([\nabla_{u},A]u - \left(\nabla u\right)^{t}Au - (\dive u) Au\Big)\right\}\circ \varphi,
\end{align*}
and $u$ solves the EPDiff equation~\eqref{eq:EPDiff}, if and only if $(\varphi,v)$ is a solution of
\begin{equation}\label{eq:spray}
  \left\{\begin{aligned}
  \varphi_{t} &= v, \\
  v_{t} &= S_{\varphi}(v),
  \end{aligned}
  \right.
\end{equation}
where
\begin{equation*}
  S_{\varphi}(v):=\left(R_{\varphi}\circ S\circ R_{\varphi^{-1}} \right)(v),
\end{equation*}
and
\begin{equation*}
  S(u):= A^{-1}\Big([\nabla_{u},A]u - \left(\nabla u\right)^{t}Au - (\dive u) Au\Big).
\end{equation*}

The \emph{second order vector field} on $\Diff(M)$ (which is a vector field on $T\Diff(M)$), defined by
\begin{equation}\label{eq:geodesic-spray}
  F : (\varphi,v)\mapsto \left(\varphi,v,v, S_{\varphi}(v)\right)
\end{equation}
is called the \emph{geodesic spray}, following Lang~\cite{Lan1999}. A geodesic is by definition an integral curve of the geodesic spray.

We will now establish that the extension and smoothness of the spray on the Hilbert manifold $T\DRd{q}$ reduces to the extension and smoothness of the twisted map
\begin{equation*}
  \varphi \mapsto R_{\varphi}\circ A\circ R_{\varphi^{-1}}, \qquad \DRd{q} \to \mathcal{L}(\HRd{q},\HRd{q-r}).
\end{equation*}

\begin{theorem}\label{thm:smoothness-spray}
  Let $r \ge 1$ and $q > 1 + d/2$, with $q \ge r$. Suppose that $A$ extends to $\mathrm{Isom}(\HRd{q},\HRd{q-r})$, and moreover that
  \begin{equation*}
    \varphi \mapsto A_{\varphi} = R_{\varphi} \circ A \circ R_{\varphi^{-1}}, \qquad \DRd{q} \to \mathcal{L}(\HRd{q},\HRd{q-r})
  \end{equation*}
  is smooth. Then the geodesic spray
  \begin{equation*}
    (\varphi, v) \mapsto S_{\varphi} (v) = R_{\varphi} \circ S \circ R_{\varphi^{-1}} (v),
  \end{equation*}
  where
  \begin{equation}\label{eq:S-map}
    S(u) = A^{-1} \left\{ [A,\nabla_{u}] u - (\nabla u )^{t} Au - (\dive u) Au \right\}
  \end{equation}
  extends smoothly to $T\DRd{q} = \DRd{q}\times \HRd{q}$. The same result holds for the torus $\Td$.
\end{theorem}

The proof given below is only sketched. The reader is invited to carefully repeat all the calculations in order to master it.

\begin{proof}[Sketch of proof]
  Set
  \begin{equation*}
    Q^{1}(u) := [A,\nabla_{u}] u, \quad Q^{2}(u) := (\nabla u )^{t} Au,
    \quad Q^{3}(u) := (\dive u) Au .
  \end{equation*}
  Then
  \begin{equation*}
    S_{\varphi}(v) = A_{\varphi}^{-1} \left\{ Q^{1}_{\varphi}(v) -
    Q^{2}_{\varphi}(v) - Q^{3}_{\varphi}(v) \right\},
  \end{equation*}
  and the proof reduces to establish, using the chain rule, that the mappings
  \begin{equation*}
    (\varphi,v) \mapsto Q^{i}_{\varphi}(v), \quad \text{and} \quad
    (\varphi,w) \mapsto A_{\varphi}^{-1} (w)
  \end{equation*}
  are smooth, for $i=1,2,3$.

  (a) From the expression of the first derivative of the twisted map $\varphi \mapsto A_{\varphi}$ (see lemma~\ref{lem:twisted-map-derivatives}), we deduce that
  \begin{equation*}
    Q^{1}_{\varphi}(v) = -\partial_{\varphi}A_{\varphi}(v,v)\,.
  \end{equation*}
  Therefore
  \begin{equation*}
    (\varphi,v) \mapsto  Q^{1}_{\varphi}(v), \quad \DRd{q}\times\HRd{q} \to \HRd{q-r}
  \end{equation*}
  is smooth.

  (b) We have
  \begin{equation*}
    Q^{2}_{\varphi}(v) = \left[\big( \nabla (v\circ\varphi^{-1}) \big)^{t} \circ \varphi\right] A_{\varphi}(v) = \frac{1}{J_{\varphi}} \mathrm{Com}(d\varphi) (dv)^{t} A_{\varphi}(v)\,,
  \end{equation*}
  where $\mathrm{Com}(d\varphi)$ is the matrix of cofactors of $d\varphi$. Thus, the smoothness of the mapping
  \begin{equation*}
    (\varphi,v) \mapsto  \left[\big( \nabla (v\circ\varphi^{-1}) \big)^{t} \circ \varphi\right] A_{\varphi}(v), \quad
    \DRd{q}\times\HRd{q} \to \HRd{q-r}
  \end{equation*}
  results from lemma~\ref{lem:pointwise-multiplication} and lemma~\ref{lem:inverse-Jacobian-smoothness}.

  (c) We have
  \begin{equation*}
    Q^{3}_{\varphi}(v) = \big( \dive(v\circ\varphi^{-1})\circ\varphi \big) A_{\varphi} (v) = \frac{1}{J_{\varphi}} \tr \left[dv (\mathrm{Com}(d\varphi))^{t}\right] A_{\varphi} (v)
  \end{equation*}
  and we conclude as in (b) that
  \begin{equation*}
    (\varphi,v) \mapsto  Q^{3}_{\varphi}(v), \quad \DRd{q}\times\HRd{q} \to \HRd{q-r}
  \end{equation*}
  is smooth.

  (d) The set
  \begin{equation*}
    \mathrm{Isom}(\HRd{q},\HRd{q-r})
  \end{equation*}
  is open in
  \begin{equation*}
    \mathcal{L}(\HRd{q},\HRd{q-r})
  \end{equation*}
  and the mapping
  \begin{gather*}
    P \mapsto P^{-1}, \\
    \mathrm{Isom}(\HRd{q},\HRd{q-r}) \to \mathcal{L}(\HRd{q-r},\HRd{q})
  \end{gather*}
  is smooth (even real analytic). Besides
  \begin{equation*}
    A_{\varphi} \in \mathrm{Isom}(\HRd{q},\HRd{q-r}),
  \end{equation*}
  for all $\varphi \in \DRd{q}$, and the mapping
  \begin{equation*}
    \varphi \mapsto A_{\varphi}, \quad \DRd{q} \to
    \mathrm{Isom}(\HRd{q},\HRd{q-r})
  \end{equation*}
  is smooth. Thus
  \begin{equation*}
    (\varphi,w) \mapsto  A^{-1}_{\varphi}(w), \quad  \HRd{q-r} \to \HRd{q}
  \end{equation*}
  is smooth.
\end{proof}


\subsection{The particle-trajectory method}
\label{subsec:part-traj}

The following reduction is a special case of a reformulation due to Ebin~\cite{Ebi1984}.

\begin{lemma}
  A smooth curve curve $\varphi(t)$ is a geodesic issued from $\id$ with initial velocity $u_{0}$ iff $\varphi(t)$ is an integral curve of Ebin's vector field
  \begin{equation}\label{eq:Ebin-field}
    X(\varphi) := A_{\varphi}^{-1}\left(\frac{1}{J_{\varphi}} (d\varphi^{-1})^{t} Au_{0} \right)
  \end{equation}
  defined on $\DiffRd$.
\end{lemma}

\begin{proof}[Sketch of proof]
  Let $\varphi(t)$ be a geodesic issued from $\id$ with initial velocity $u_{0}$ and let $u(t) = \varphi_{t} \circ \varphi^{-1}$ be the Eulerian velocity. Then, $u(t)$ is a solution of the Euler--Arnold equation~\eqref{eq:Diff-Euler-Arnold}. Moreover, the general conservation law~\eqref{eq:momentum-conservation} specialized for $\DiffRd$ may be written as
  \begin{equation*}
    J_{\varphi}(d\varphi)^{t}. (m \circ \varphi) = m_{0},
  \end{equation*}
  where $m = Au$. We get thus
  \begin{equation*}
    \varphi_{t} = u \circ \varphi = A_{\varphi}^{-1}\left(\frac{1}{J_{\varphi}} (d\varphi^{-1})^{t} Au_{0} \right),
  \end{equation*}
  where $A_{\varphi} := R_{\varphi}\circ A\circ R_{\varphi^{-1}}$.
\end{proof}

We will now establish that the extension and smoothness of the vector field $X$ on the Hilbert manifold $\DRd{q}$ reduces to the extension and smoothness of the twisted map
\begin{equation*}
  \varphi \mapsto R_{\varphi}\circ A\circ R_{\varphi^{-1}}, \qquad \DRd{q} \to \mathcal{L}(\HRd{q},\HRd{q-r}).
\end{equation*}

thm:smoothness-spray\begin{theorem}\label{thm:smoothness-X}
  Let $r \ge 1$ and $q > 1 + d/2$, with $q \ge r$. Suppose that $A$ extends to $\mathrm{Isom}(\HRd{q},\HRd{q-r})$, and moreover that
  \begin{equation*}
    \varphi \mapsto A_{\varphi} = R_{\varphi} \circ A \circ R_{\varphi^{-1}}, \qquad \DRd{q} \to \mathcal{L}(\HRd{q},\HRd{q-r})
  \end{equation*}
  is smooth. Then the vector field
  \begin{equation*}
    \varphi \mapsto X(\varphi) := A_{\varphi}^{-1}\left(\frac{1}{J_{\varphi}} (d\varphi^{-1})^{t} Au_{0} \right)
  \end{equation*}
  extends smoothly to $\DRd{q}$ for each $u_{0} \in \HRd{q}$.
\end{theorem}

\begin{proof}
  Note first that we have
  \begin{equation*}
    \frac{1}{J_{\varphi}} (d\varphi^{-1})^{t} Au_{0} = \frac{1}{J_{\varphi}^{2}} \mathrm{Com}(d\varphi) Au_{0}
  \end{equation*}
  and the smoothness of the mapping
  \begin{equation*}
    \varphi \mapsto \frac{1}{J_{\varphi}^{2}} \mathrm{Com}(d\varphi) Au_{0}, \qquad \DRd{q} \to \HRd{q-r}
  \end{equation*}
  results from lemma~\ref{lem:pointwise-multiplication} and lemma~\ref{lem:inverse-Jacobian-smoothness}. Thus the problem of the smoothness of the vector field
  \begin{equation*}
    X(\varphi) = \varphi \mapsto A_{\varphi}^{-1}\left(\frac{1}{J_{\varphi}^{2}} \mathrm{Com}(d\varphi) Au_{0}\right)
  \end{equation*}
  reduces to show that the mapping
  \begin{equation*}
    (\varphi,v) \mapsto A^{-1}_{\varphi}(v), \qquad \DRd{q}\times\HRd{q-r} \to \HRd{q}
  \end{equation*}
  is smooth. Now, by corollary~\ref{cor:smoothness-equivalence}, this last assertion is equivalent to show the smoothness of the mapping
  \begin{equation*}
    \varphi \mapsto A^{-1}_{\varphi}, \qquad \DRd{q} \to \mathcal{L}(\HRd{q-r},\HRd{q})\,,
  \end{equation*}
  and since $A_{\varphi}^{-1} \in \mathrm{Isom}(\HRd{q-r},\HRd{q})$, this is equivalent to show the smoothness of
  \begin{equation*}
    \varphi \mapsto A_{\varphi}, \qquad \DRd{q} \to \mathcal{L}(\HRd{q},\HRd{q-r})\,.
  \end{equation*}
\end{proof}


\section{Smoothness of the twisted map}
\label{sec:twisted-map}

Let $A$ be a continuous linear operator from $\HRd{\infty}$ to itself (a similar discussion holds for $\HTd{\infty}$) and consider the twisted map
\begin{equation*}
  A_{\varphi} := R_{\varphi}A R_{\varphi^{-1}},
\end{equation*}
where $R_{\varphi}v = v \circ \varphi$ and $\varphi \in \DiffRd$. Since $\DiffRd$ is a Fr\'{e}chet Lie group with Lie algebra $\HRd{\infty}$, the mapping
\begin{equation}\label{eq:twisted_map}
  (\varphi,v) \mapsto A_{\varphi}v, \qquad \DiffRd \times \HRd{\infty} \to \HRd{\infty}
\end{equation}
is smooth. The aim of this section is to prove that the twisted map
\begin{equation*}
  \varphi \mapsto A_{\varphi} := R_{\varphi}A R_{\varphi^{-1}},
\end{equation*}
extends smoothly from $\DRd{q}$ to $\mathcal{L}(\HRd{q},\HRd{q-r})$ for $q > 1 + d/2$ and $q \ge r$, when $A$ is a Fourier multiplier of class $S^{r}$ with $r \ge 1$.

Note that the problem is not trivial, because even if $A$ extends to a bounded linear map in $\mathcal{L}(\HRd{q},\HRd{q-r})$, the mapping
\begin{equation*}
  (\varphi,v) \mapsto R_{\varphi} (v), \quad \DRd{q}\times\HRd{q} \to \HRd{q}
\end{equation*}
is \emph{not differentiable} (see \cite{EM1970,Mic2006} for instance), and the mapping $\varphi \mapsto R_{\varphi}$ is not even continuous (see~\cite[Remark B.4]{EK2014}).

\begin{remark}
  When $A$ is a differential operator of order $r$, $A_{\varphi}$ is a linear differential operator whose coefficients are polynomial expressions of $1/J_{\varphi}$ and the derivatives of $\varphi$ up to order $r$ (see~\cite{EM1970,EK2011} for instance). For example, in one dimension, $D_{\varphi} = (1/\varphi_x)D$, where $D := d/dx$. In that case, $\varphi \mapsto A_{\varphi}$ is smooth (in fact real analytic) in the Hilbert setting. But this observation is no longer true for a general Fourier multiplier of class $S^{r}$ and more work is required.
\end{remark}

The strategy is the following. Given a continuous linear operator
\begin{equation*}
  A : \HRd{\infty} \to \HRd{\infty},
\end{equation*}
the expressions of the G\^{a}teaux derivatives of the twisted map $A_{\varphi}$ can be computed in the Fr\'{e}chet setting: they are given by the multilinear operators $A_{n,\varphi}$ (lemma~\ref{lem:twisted-map-derivatives}). Then it is shown that the mapping $\varphi \mapsto A_{\varphi}$ is smooth in the Hilbert setting if and only if the $A_{n} := A_{n,\id}$ are bounded for the Hilbert norm (lemma~\ref{lem:smoothness-lemma}). Finally, it is established that if $A$ is a Fourier multiplier in the class $\mathcal{S}^{r}$, then the $A_{n}$ are bounded and so $A_{\varphi}$ is smooth in the Hilbert setting (theorem~\ref{thm:twisted-map-smoothness}). The proof of this theorem uses a crucial commutator estimate (lemma~\ref{lem:boundedness-lemma}) which can be considered as a multilinear generalization of the Kato--Ponce estimate~\cite{KP1988}. The details below are given in the non-compact case $\Rd$ but apply also to the compact case of the torus $\Td$ with slight modifications which will be notified.

The first step is to obtain an explicit formula for the $n$-th partial G\^{a}teaux derivative of the twisted map~\eqref{eq:twisted_map} in the Fr\'{e}chet setting. The calculation was done in~\cite[Lemma 3.2]{EK2014} and will not be repeated here.

\begin{lemma}\label{lem:twisted-map-derivatives}
  Let $A$ be a continuous linear operator from $\HRd{\infty}$ to itself. Then
  \begin{equation*}
    \partial^{n}_{\varphi}A_{\varphi} (v,\delta\varphi_{1}, \dotsc ,\delta\varphi_{n}) = R_{\varphi}A_{n}R_{\varphi}^{-1}(v,\delta\varphi_{1}, \dotsc ,\delta\varphi_{n}),
  \end{equation*}
  where
  \begin{equation*}
    A_{n}:= \partial^{n}_{\id}A_{\varphi} \in \mathcal{L}^{n+1}(\HRd{\infty},\HRd{\infty})
  \end{equation*}
  is the $(n+1)$-linear operator defined inductively by $A_{0} = A$ and
  \begin{multline}\label{eq:recurrence-relation}
    A_{n+1}(u_{0},u_{1}, \dotsc , u_{n+1}) = \nabla_{u_{n+1}} \left(A_{n}(u_{0}, u_{1}, \dotsc , u_{n}) \right)\\
    - \sum_{k=0}^{n} A_{n}(u_{0}, u_{1}, \dotsc ,\nabla_{u_{n+1}} u_{k}, \dotsc , u_{n}),
  \end{multline}
  where $\nabla$ is the canonical covariant derivative on $\RR^{d}$.
\end{lemma}

It may be useful to rather think of $A_{n}$ as a $n$-linear mapping
\begin{equation*}
  \HRd{\infty} \times \dotsb \times \HRd{\infty} \to \mathcal{L}(\HRd{\infty},\HRd{\infty})
\end{equation*}
and write
\begin{equation*}
  A_{n}(u_{0},u_{1}, \dotsc , u_{n}) = A_{n}(u_{1}, \dotsc , u_{n})u_{0}.
\end{equation*}
The recurrence relation~\eqref{eq:recurrence-relation} rewrites then accordingly as:
\begin{multline}\label{eq:def-Rec}
  \Rec (A_{n})(u_{1}, \dotsc , u_{n+1}) := [\nabla_{u_{n+1}}, A_{n}(u_{1}, \dotsc , u_{n})]\\
  - \sum_{k=1}^{n} A_{n}(u_{1}, \dotsc ,\nabla_{u_{n+1}} u_{k}, \dotsc , u_{n}).
\end{multline}

\begin{example}\label{ex:A1-A2}
  For $n=1$, we have
  \begin{equation*}
    A_{1}(u_{1}) = [\nabla_{u_{1}},A],
  \end{equation*}
  and for $n=2$, we get
  \begin{equation*}
    A_{2}(u_{1},u_{2}) = [\nabla_{u_{2}},[\nabla_{u_{1}},A]] - [\nabla_{\nabla_{u_{2}}u_{1}},A].
  \end{equation*}
\end{example}

\begin{remark}
  When $d=1$ and $A$ commutes with $D := d/dx$, the following nice formula for $A_{n}$ was obtained in~\cite{Cis2015}:
  \begin{equation}\label{eq:Cismas-formula}
    A_{n}(u_{1}, \dotsc , u_{n}) = [u_{1},[u_{2},[ \dotsc [ u_{n}, D^{n-1}A]\dotsc]]]D, \qquad n \ge 1 \,.
  \end{equation}
\end{remark}

The next lemma states that the smoothness of $A_{\varphi}$ in the Hilbert setting reduces to the boundedness of the $A_{n}$. It is a slight extension of~\cite[Theorem 3.4]{EK2014} and we redirect to this reference for a full proof.

\begin{lemma}[Smoothness Lemma]\label{lem:smoothness-lemma}
  Let
  \begin{equation*}
    A : \HRd{\infty} \to \HRd{\infty}
  \end{equation*}
  be a continuous linear operator. Given $q > 1 + d/2$ and $0 \le p_{1}, p_{2} \le q$, suppose that $A$ extends to a bounded operator from $\HRd{p_{1}}$ to $\HRd{p_{2}}$. Then
  \begin{equation*}
    \varphi \mapsto A_{\varphi}:= R_{\varphi} \circ A \circ R_{\varphi^{-1}} , \quad \DRd{q} \to \mathcal{L}(\HRd{p_{1}},\HRd{p_{2}})
  \end{equation*}
  is smooth, if and only if, each operator $A_{n}$ defined by \eqref{eq:def-Rec}, extends to a bounded $n$-linear operator in $\mathcal{L}^{n}(\HRd{q}, \mathcal{L}(\HRd{p_{1}},\HRd{p_{2}}))$.
\end{lemma}

The idea of the proof of lemma~\ref{lem:smoothness-lemma}, which is inductive, is the following. First, we show that if $A_{n}$ is bounded, then the mapping
\begin{equation*}
  \varphi \mapsto A_{n,\varphi} := R_{\varphi}A_{n}R_{\varphi^{-1}}, \quad \DRd{q} \to \mathcal{L}^{n+1}(\HRd{p_{1}},\HRd{p_{2}})
\end{equation*}
is locally bounded. Then, we prove that if $A_{n+1,\varphi}$ is locally bounded, then $A_{n,\varphi}$ is locally Lipschitz. Finally we show that if $A_{n+1,\varphi}$ is continuous, then $A_{n,\varphi}$ is $C^{1}$ and its derivative is $A_{n+1,\varphi}$. The proof, given below, uses the following two elementary lemmas, which will be stated without proof.

\begin{lemma}\label{lem:path_integral_continuity}
  Let $X$ be a topological space and $E$ a Banach space. Let $f: [0,1] \times X\to E$ be a continuous mapping. Then,
  \begin{equation*}
    g(x):= \int_{0}^{1} f(t,x)\, dt
  \end{equation*}
  is continuous.
\end{lemma}

\begin{lemma}\label{lem:mean_value_criteria}
  Let $E$, $F$ be Banach spaces and $U$ a convex open set in $E$. Let $\alpha: U \to \mathcal{L}(E,F)$ be a continuous mapping and $f : U \to F$ a mapping such that
  \begin{equation*}
    f(y) - f(x) = \int_{0}^{1} \alpha(ty + (1-t)x)(y-x)\, dt,
  \end{equation*}
  for all $x,y \in U$. Then $f$ is $C^{1}$ on $U$ and $df = \alpha$.
\end{lemma}

\begin{proof}[Proof sketch of lemma~\ref{lem:smoothness-lemma}]
  If
  \begin{equation*}
    \varphi \mapsto A_{\varphi}, \quad \DRd{q} \to \mathcal{L}(\HRd{p_{1}},\HRd{p_{2}})
  \end{equation*}
  is smooth, then $A_{n}$, which is the $n$-th Fr\'{e}chet derivative $\partial^{n}_{id}A_{\varphi}$ at the identity is bounded.

  Conversely, suppose that each $A_{n}$ extends to a bounded $n$-linear operator in $\mathcal{L}^{n}(\HRd{q}, \mathcal{L}(\HRd{p_{1}},\HRd{p_{2}}))$. Then, by lemma~\ref{lem:composition}, the mapping
  \begin{equation*}
    \varphi \mapsto A_{n,\varphi}, \qquad \DRd{q} \to \mathcal{L}^{n}(\HRd{q}, \mathcal{L}(\HRd{p_{1}},\HRd{p_{2}}))
  \end{equation*}
  defined by
  \begin{equation*}
    A_{n,\varphi}(v_{1}, \dotsc , v_{n})v_{0} := \left(A_{n}(v_{1}\circ\varphi^{-1}, \dotsc , v_{n}\circ\varphi^{-1})v_{0}\circ\varphi^{-1}\right)\circ\varphi
  \end{equation*}
  is \emph{locally bounded}. We will now show that $A_{\varphi}$ is $C^{1}$. By reasoning inductively, we can prove the same way that $A_{\varphi}$ is smooth.

  Because $\varphi \mapsto A_{2,\varphi}$ is \emph{locally bounded}, it is possible to find around each diffeomorphism $\psi \in \DRd{q}$ a convex open set $U$ in $\DRd{q}$ and a positive constant $K$ such that
  \begin{equation*}
    \norm{A_{2,\varphi}} \le K, \quad \forall \varphi \in U.
  \end{equation*}
  Let $\varphi_{0},\varphi_{1}$ in $\DiffRd \cap U$ and set $\varphi(t):=(1-t)\varphi_{0}+t\varphi_{1}$ for $t\in [0,1]$. Given $v_{0},v_{1}\in\HRd{\infty}$ with $\norm{v_{0}}_{H^{q}}, \norm{v_{1}}_{H^{p_{1}}}\le 1$, we obtain, using the mean value theorem in $\HRd{\infty}$ that
  \begin{equation*}
    A_{1, \varphi_{1}}(v_{0},v_{1}) - A_{1, \varphi_{0}}(v_{0},v_{1}) =
    \int_{0}^{1} A_{2, \varphi(t)}(v_{0},v_{1}, \varphi_{1}-\varphi_{0} )\, dt,
  \end{equation*}
  and thus that
  \begin{equation*}
    \norm{A_{1, \varphi_{1}}(v_{0},v_{1}) - A_{1, \varphi_{0}}(v_{0},v_{1})}_{H^{p_{2}}} \le
    K \norm{ \varphi_{1}-\varphi_{0}}_{H^{q}},
  \end{equation*}
  for all $v_{0},v_{1}\in\HRd{\infty}$ with $\norm{v_{0}}_{H^{q}}, \norm{v_{1}}_{H^{p_{1}}}\le 1$. The assertion that $A_{1, \varphi}$ is Lipschitz continuous on $U$ follows from lemma~\ref{lem:path_integral_continuity} and the density of the space $\HRd{\infty}$ in the Sobolev spaces $\HRd{p}$. Now, we have
  \begin{equation*}
    A_{\varphi_{1}}(v) - A_{\varphi_{0}}(v) = \int_{0}^{1} A_{1, \varphi(t)}(v,\varphi_{1}-\varphi_{0})\, dt,
  \end{equation*}
  for all $\varphi_{0},\varphi_{1}\in U\cap\HRd{\infty}$ and $v \in \HRd{\infty}$ and we conclude similarly that this formula is still true for all $\varphi_{0}, \varphi_{1} \in U$ and $v\in \HRd{p_{1}}$. Therefore, we can write in $\mathcal{L}(\HRd{p_{1}},\HRd{p_{2}})$
  \begin{equation*}
    A_{\varphi_{1}} - A_{\varphi_{0}} = \int_{0}^{1} A_{1, \varphi(t)}(\varphi_{1}-\varphi_{0})\, dt,
  \end{equation*}
  and, by virtue of lemma~\ref{lem:mean_value_criteria}, we conclude that $\varphi \mapsto A_{\varphi}$ is $C^{1}$ and that $dA_{\varphi} = A_{1,\varphi}$.
\end{proof}

Lemma~\ref{lem:smoothness-lemma} has the following interesting corollary.

\begin{corollary}\label{cor:smoothness-equivalence}
  Let $A:\HRd{\infty} \to \HRd{\infty}$ be a continuous linear operator and $A_{\varphi} = R_{\varphi} \circ A \circ R_{\varphi^{-1}}$. Let $q > 1 + d/2$ and $0 \le p_{1}, p_{2} \le q$. Then, the following assertions are equivalent:
  \begin{enumerate}
    \item There exists a smooth extension
          \begin{equation*}
            (\varphi,v) \mapsto A_{\varphi}(v), \qquad \DRd{q}\times\HRd{p_{1}} \to \HRd{p_{2}}
          \end{equation*}
    \item There exists a smooth extension
          \begin{equation*}
            \varphi \mapsto A_{\varphi}, \qquad \DRd{q} \to \mathcal{L}(\HRd{p_{1}},\HRd{p_{2}}),
          \end{equation*}
  \end{enumerate}
\end{corollary}

\begin{proof}
  The fact that (2) implies (1) is trivial and results from the fact that
  \begin{equation*}
    (P,v) \mapsto P(v), \quad \mathcal{L}(\HRd{p_{1}},\HRd{p_{2}}) \times \HRd{p_{1}} \to \HRd{p_{2}}
  \end{equation*}
  is continuous. Conversely, suppose that
  \begin{equation*}
    (\varphi,v) \mapsto A_{\varphi}(v), \qquad \DRd{q}\times\HRd{p_{1}} \to \HRd{p_{2}}
  \end{equation*}
  is smooth. Then, in particular, for each $n \ge 1$,
  \begin{equation*}
    v \mapsto A_{n}(v,\cdot, \dotsc , \cdot), \qquad \HRd{p_{1}} \to \mathcal{L}^{n}(\HRd{q},\HRd{p_{2}})
  \end{equation*}
  is bounded. Therefore each $A_{n}$ is bounded as an operator in
  \begin{equation*}
    \mathcal{L}^{n}(\HRd{q}, \mathcal{L}(\HRd{p_{1}},\HRd{p_{2}}))
  \end{equation*}
  and by lemma~\ref{lem:smoothness-lemma} the mapping
  \begin{equation*}
    \varphi \mapsto A_{\varphi}, \qquad \DRd{q} \to \mathcal{L}(\HRd{p_{1}},\HRd{p_{2}}),
  \end{equation*}
  is smooth.
\end{proof}

Let $A : \HRd{\infty} \to \HRd{\infty}$ be a continuous linear operator which is \emph{translation invariant}. Then, it can be written (see for instance~\cite[Part II, Section 2.1]{RT2010} ) as
\begin{equation*}
  (Au)(x) = \int_{\Rd} e^{2\pi i x\cdot \xi} a(\xi) \hat{u}(\xi) \,d\xi \,,
\end{equation*}
where $a: \Rd \to \mathcal{L}(\CC^{d})$. Such an operator, also noted $a(D)$ or $\op{a(\xi)}$, is called a \emph{Fourier multiplier with symbol $a$}.

\begin{remark}
  A necessary and sufficient condition for $A$ to be \emph{$L^{2}$-symmetric and positive definite} is that $a(\xi)$ is \emph{Hermitian and positive definite} for all $\xi \in \Rd$.
\end{remark}

Of course, some regularity conditions are required on the symbol $a$ to insure that the operator is well-defined. In the following, we will restrict ourselves to a class of symbols for which this operation is well-defined on $\HRd{\infty}$ and leads to operators with nice properties.

\begin{definition}\label{def:class-Sr}
  Given $r \in \RR$, a Fourier multiplier $a(D)$ is of class $S^{r}$ iff $a\in C^\infty(\Rd,\mathcal{L}(\CC^{d}))$ and for each $\alpha \in \NN^{d}$ there exists a positive constant $C_{a,\alpha}$ such that
  \begin{equation*}
    \abs{\partial^{\alpha}a(\xi)} \le C_{a,\alpha} \left\langle \xi \right\rangle^{r-\abs{\alpha}}\,,
  \end{equation*}
  where $\abs{\alpha} := \alpha_{1} + \dotsb + \alpha_{d}$ and $\left\langle \xi \right\rangle := \left( 1 + \abs{\xi}^{2}\right)^{1/2}$.
\end{definition}

\begin{example}
  Any linear differential operator of order $r$ with constant coefficients is in this class. Furthermore, $\Lambda^{r} := \op{\left\langle \xi \right\rangle^{r}}$ belongs to this class.
\end{example}

\begin{remark}
  Note that a Fourier multiplier $a(D)$ of class $S^{r}$ extends to a \emph{bounded linear operator}
  \begin{equation*}
    \HRd{q} \to \HRd{q-r}
  \end{equation*}
  for any $q\in\RR$ and $r \ge q$, and hence on $\HRd{\infty}$.
\end{remark}

For a Fourier multiplier on the torus $\Td$, the symbol $a$ is a function
\begin{equation*}
  a: \ZZ^{d} \to \mathcal{L}(\CC^{d})
\end{equation*}
and the definition of the class $S^{r}$ should be adapted. We define therefore the \emph{discrete partial derivative}, or \emph{partial difference operator}, as
\begin{equation*}
  (\triangle_{j} a)(\xi) := a(\xi + e_{j}) - a(\xi) \,,
\end{equation*}
where $(e_{j})$ is the canonical basis of $\ZZ^{d}$. We extend then the definition to define $\triangle^{\alpha}$ for a multi-index $\alpha \in \NN^{d}$ as
\begin{equation*}
  \triangle^{\alpha} := \triangle_{1}^{\alpha_{1}} \dotsb \triangle_{d}^{\alpha_{d}} \,.
\end{equation*}

\begin{definition}\label{def:toroidal-class-Sr}
  Given $r \in \RR$, a Fourier multiplier $a(D)$ is of \emph{toroidal class} $S^{r}$ iff for each $\alpha \in \NN^{d}$ there exists a positive constant $C_{a,\alpha}$ such that
  \begin{equation*}
    \abs{\triangle^{\alpha} a(\xi)} \le C_{a,\alpha} \left\langle \xi \right\rangle^{r-\abs{\alpha}}\,.
  \end{equation*}
\end{definition}

\begin{example}
  Any linear differential operator of order $r$ with constant coefficients on $C^{\infty}(\Td)$ as well as $\Lambda^{r} := \op{\left\langle \xi \right\rangle^{r}}$ are in the toroidal class $S^{r}$. The operator $\op{\abs{\xi}}$, where $\abs{\xi} = \abs{\xi^{1}} + \dotsb + \abs{\xi^{d}}$, belongs to the toroidal class $S^{1}$.
\end{example}

We will now prove the following theorem.

\begin{theorem}\label{thm:twisted-map-smoothness}
  Let $A = a(D)$ be a Fourier multiplier of class $S^{r}$, with $r \ge 1$. Then the mapping $\varphi \mapsto A_{\varphi} := R_{\varphi^{-1}}A R_{\varphi}$ extends smoothly to
  \begin{equation*}
    \DRd{q} \to \mathcal{L}(\HRd{q},\HRd{q-r})
  \end{equation*}
  for $q > 1 + d/2$ and $r \le q$. The same result holds for the torus $\Td$ if $A$ is a Fourier multiplier with symbol in the toroidal class $S^{r}$.
\end{theorem}

The first proof of theorem~\ref{thm:twisted-map-smoothness} for the special case of the circle was obtained in~\cite{EK2014}. This proof was then extended to the case of $\Rd$ for $d \ge 1$ in~\cite{BEK2015}. We sketch here a slightly alternative proof, hopefully simpler. It is based on the following lemma.

\begin{lemma}\label{lem:boundedness-lemma}
  Let $P$ be a Fourier multiplier of class $S^{r+n-1}$. Given $w \in \HRd{\infty}$ and $f_{1}, \dotsc , f_{n} \in \HR{\infty}$, we have
  \begin{equation*}
    \norm{S_{n,P}(f_{1},\dotsc , f_{n})w}_{H^{q-r}} \lesssim \norm{f_{1}}_{H^{q}} \dotsb \norm{f_{n}}_{H^{q}} \norm{w}_{H^{q-1}}
  \end{equation*}
  for $q > 1 + d/2$ and $1 \le r \le q$, where
  \begin{equation*}
    S_{n,P}(f_{1},f_{2}, \dotsc , f_{n}) := [f_{1},[f_{2},[ \dotsc [ f_{n}, P]\dotsc]]]\,.
  \end{equation*}
\end{lemma}

\begin{proof}
  Let $P = \op{p}$. Then, the Fourier transform of $S_{n,P}(f_{1},\dotsc , f_{n})w$ can be written as
  \begin{equation}\label{eq:integral-formula-for-Sn}
    \widehat{S_{n,P}}(\xi) = \int_{\xi_{0} + \dotsb + \xi_{n} = \xi} \hat{f}_{1}(\xi_{1}) \dotsm \hat{f}_{n}(\xi_{n}) \left[ p_{n}(\xi_{0}, \dotsc ,\xi_{n}). \hat{w}(\xi_{0}) \right] \, d\mu
  \end{equation}
  where $d\mu$ is the Lebesgue measure on the subspace $\xi_{0} + \dotsb + \xi_{n} = \xi$ of $(\RR^{d})^{n+1}$ and
  \begin{equation*}
    p_{n}(\xi_{0},\xi_{1}, \dotsc ,\xi_{n}) := \sum_{J \subset \set{1, \dotsc , n}}(-1)^{\abs{J}} p\left(\xi_{0} + \sum_{j\in J}\xi_{j} \right) \,.
  \end{equation*}
  Now, the observation that the sequence $p_{n}$ satisfies the recurrence relation
  \begin{equation*}
    p_{k+1}(\xi_{0},\dotsc,\xi_{k+1}) = p_{k}(\xi_{0},\dotsc,\xi_{k}) - p_{k}(\xi_{0}+\xi_{k+1}, \xi_{1},\dots,\xi_{k})\,,
  \end{equation*}
  and the iterative application of the mean value theorem leads to the following estimate:
  \begin{equation}\label{eq:pn-estimate}
    \begin{aligned}
      \abs{p_{n}(\xi_{0},\dotsc,\xi_{n})} & \le \left( \prod_{j=1}^{n} \abs{\xi_{j}} \right) \cdot \sup_{\xi \in K_{n}} \abs{d^{n} p(\xi)}                                                    \\
                                          & \le C_{p,n} \left( \prod_{j=1}^{n} \left\langle \xi_{j} \right\rangle \right) \cdot \sup_{\xi \in K_{n}} \left\langle \xi \right\rangle^{r-1} \,,
    \end{aligned}
  \end{equation}
  where $K_{n}$ is the convex hull of the points $\xi_{0} + \sum_{j \in J} \xi_j$, where $J$ is any subset of $\{1,\dotsc,n\}$ (see~\cite{BEK2015} or~\cite{EK2014} for the details). For $r \geq 1$, the function $\xi \mapsto \langle \xi \rangle^{r-1}$ attains its maximum on $K_{n}$ at one of the extremal points $\xi_{0} + \sum_{j \in J} \xi_{j}$. Hence
  \begin{equation*}
    \sup_{\xi \in K_{n}} \left\langle \xi \right\rangle^{r-1} \leq
    \sum_{J \subseteq \set{1, \dotsc , n}} \left\langle \xi_{0} + \sum_{j \in J} \xi_{j} \right\rangle^{r-1}\,.
  \end{equation*}
  We have therefore
  \begin{multline*}
    \abs{\widehat{S_{n,P}}(\xi)} \le C_{p,n} \sum_{J \subseteq \set{1, \dotsc , n}}
    \\
    \int_{\xi_{0} + \dotsb +\xi_{n} = \xi} \left\langle \xi_{0} + \sum_{j \in J} \xi_{j} \right\rangle^{r-1} \left( \prod_{j=1}^{n} \langle \xi_{j} \rangle \right) \abs{\hat{f}_{1}(\xi_{1}) \dotsb \hat{f}_{n}(\xi_{n})} \abs{\hat{w}(\xi_{0})} \, d\mu \,.
  \end{multline*}
  But
  \begin{multline*}
    \int_{\xi_{0} + \dotsb +\xi_{n} = \xi} \left\langle \xi_{0} + \sum_{j \in J} \xi_{j} \right\rangle^{r-1} \left( \prod_{j=1}^{n} \langle \xi_{j} \rangle \right) \abs{\hat{f}_{1}(\xi_{1}) \dotsb \hat{f}_{n}(\xi_{n})} \abs{\hat{w}(\xi_{0})} \, d\mu
    \\
    = \mathfrak{F} \left( \left[\prod_{j \in J^{c}} \Lambda^{1}(\tilde{f}_{j})\right] \Lambda^{r-1} \left[\tilde{w} \prod_{k \in J} \Lambda^{1}(\tilde{f}_{k}) \right] \right) (\xi) \,,
  \end{multline*}
  where $\Lambda^{s}$ is the Fourier multiplier with symbol $\langle \xi \rangle^{s}$, $\tilde{f}_{j} := \mathfrak{F}^{-1}\left(\abs{\langle \xi_{j} \rangle \hat{f}_{j}}\right)$ and $\tilde{w} := \mathfrak{F}^{-1}\left(\abs{\hat{w}}\right)$. We have thus, using the Plancherel identity
  \begin{multline*}
    \norm{S_{n,P}(f_{1},\dotsc , f_{n})w}_{H^{q-r}} = \norm{\langle \xi \rangle^{q-r} \mathfrak{F}( S_{n,P}(f_{1},\dotsc , f_{n})w)}_{L^{2}} \\
    \le C_{p,n} \sum_{J \subseteq \set{1, \dotsc , n}} \norm{ \left[\prod_{j \in J^{c}} \Lambda^{1}(\tilde{f}_{j})\right] \Lambda^{r-1} \left[\tilde{w} \prod_{k \in J} \Lambda^{1}(\tilde{f}_{k}) \right] }_{H^{q-r}} \,.
  \end{multline*}
  By lemma~\ref{lem:pointwise-multiplication}, we have moreover
  \begin{multline*}
    \norm{ \left[\prod_{j \in J^{c}} \Lambda^{1}(\tilde{f}_{j})\right] \Lambda^{r-1} \left[\tilde{w} \prod_{k \in J} \Lambda^{1}(\tilde{f}_{k}) \right] }_{H^{q-r}}
    \\
    \lesssim \norm{ \prod_{j \in J^{c}}\Lambda^{1}(\tilde{f}_{j})}_{H^{q-1}} \norm{\Lambda^{r-1} \left[\tilde{w} \prod_{k \in J} \Lambda^{1}(\tilde{f}_{k})\right] }_{H^{q-r}}\,,
  \end{multline*}
  which is bounded by $\norm{f_{1}}_{H^{q}} \dotsb \norm{f_{n}}_{H^{q}} \norm{w}_{H^{q-1}}$. This achieves the proof.
\end{proof}

\begin{remark}
  Lemma~\ref{lem:boundedness-lemma} is also true for a Fourier multiplier of toroidal class $S^{r+n-1}$ on the torus $\Td$. The only difference is that we can no longer use the mean value theorem to establish estimate~\eqref{eq:pn-estimate}. In that case, we shall use the following \emph{discrete version of the mean value theorem}:
  \begin{equation*}
    \abs{a(\eta + \xi) - a(\eta)} \le \abs{\xi} \max_{1 \le i \le d} \left(\max_{0 \le k < \abs{\xi^{i}}} \abs{\triangle_{i} a (\eta + ke_{i})}\right) \,,
  \end{equation*}
  where $\eta, \xi \in \ZZ^{d}$ and $\abs{\xi} := \abs{\xi^{1}} + \dotsb + \abs{\xi^{d}}$.
\end{remark}

\begin{proof}[Proof of Theorem~\ref{thm:twisted-map-smoothness}]
  Due to lemma~\ref{lem:smoothness-lemma}, it is enough to show that each $A_{n}$ is bounded. Note first that $A_{1}$ can be written as
  \begin{equation*}
    A_{1}(u_{1}) = \sum_{j=1}^{d} [u_{1}^{j},A]D_{j}\,,
  \end{equation*}
  where $D_{j} := \partial/\partial x^{j}$. Now, using combinatorial properties of commutators and the recurrence relation~\ref{eq:def-Rec} (see~\cite{Cis2015} for the details), we can show that for each $n \ge 1$, $A_{n}$ is a finite sum of terms
  \begin{equation}\label{eq:generic-term}
    Q_{n}(f_{1},\dotsc , f_{n}) := S_{p,D^{\alpha}A}(f_{1},\dotsc , f_{p})\partial_{i_{p+1}}f_{p+1} \dotsm \partial_{i_{n}}f_{n}D_{j}
  \end{equation}
  where
  \begin{equation*}
    D^{\alpha}:= \partial_{x_{1}}^{\alpha_{1}} \dotsm \partial_{x_{d}}^{\alpha_{d}}, \quad \abs{\alpha} = p-1, \quad p \ge 1 \,,
  \end{equation*}
  and $(f_{1},f_{2}, \dotsc , f_{n})$ stands for a permutation $(u_{\sigma(1)}^{k_{1}}, \dotsc , u_{\sigma(n)}^{k_{n}})$ of some components of $(u_{1}, \dotsc , u_{n})$. Then, using lemma~\ref{lem:boundedness-lemma}, we get
  \begin{align*}
    \norm{Q_{n}(f_{1},\dotsc , f_{n})w}_{H^{q-r}} & \lesssim \norm{f_{1}}_{H^{q}} \dotsb \norm{f_{p}}_{H^{q}}\norm{\partial_{i_{p+1}}f_{p+1} \dotsm \partial_{i_{n}}f_{n}D_{j}w}_{H^{q-1}}
    \\
                                                  & \lesssim \norm{f_{1}}_{H^{q}} \dotsb \norm{f_{n}}_{H^{q}} \norm{w}_{H^{q}}
  \end{align*}
  because $\HR{q-1}$ is a multiplicative algebra. This achieves the proof.
\end{proof}

\begin{corollary}
  Let $A = \mathbf{a}(D)$ belong to the class $\mathcal{S}^{r}$. Then the metric
  \begin{equation*}
    G_{\varphi}(v_{1},v_{2}) = \int_{\Rd}  \left( A_\varphi v_{1} \cdot v_{2} \right) \varphi^{*} d\mu \,,
  \end{equation*}
  extends to a smooth Riemannian metric on $\DRd{q}$ for $q > 1 + d/2$ and $r \le q$. The same result holds for the torus $\Td$ if $A$ is a Fourier multiplier with symbol in the toroidal class $S^{r}$.
\end{corollary}

\begin{remark}
  This applies, in particular, to $H^{s}$-metrics on $\DiffRd$ or $\DiffTd$ where $s \in \RR$ and $s \ge 1/2$. The Constantin--Lax--Majda equation corresponds to the inertia operator $A = \op{\abs{k}}$ in the toroidal class $\mathcal{S}^{1}$ and the Euler--Weil--Peterson equation to the inertia operator $A = \op{\abs{k}(k^{2}-1)}$ in the toroidal class $\mathcal{S}^{3}$.
\end{remark}


\section{Local well-posedness in the smooth category}
\label{sec:local-well-posedness}

A simple criteria which ensures that the inertia operators induces a \emph{bounded isomorphism} between $\HRd{q}$ and $\HRd{q-r}$ for all $q\in \RR$ big enough is provided by an \emph{ellipticity condition} on $A$. In that case, both the \emph{spray method} (see subsection~\ref{subsec:spray}) and the \emph{particle trajectory method} (see subsection~\ref{subsec:part-traj}) will lead to a local well-posedness for the geodesic flow. We will adopt the following definition.

\begin{definition}
  A Fourier multiplier $a(D)$ in the class $\mathcal{S}^{r}$ is called \emph{elliptic} if $a(\xi)\in\mathrm{GL}(\CC^{d})$ for all $\xi\in\Rd$ and moreover
  \begin{equation*}
    \norm{[a(\xi)]^{-1}} \lesssim \left( 1 + \abs{\xi}^{2}\right)^{-r/2}, \qquad \forall \xi \in \Rd.
  \end{equation*}
\end{definition}

The local existence of geodesics on the Hilbert manifold $T\DRd{q}$ follows then from the Cauchy-Lipschitz theorem, due to
the smoothness of the extended spray
\begin{equation*}
  F_{q}(\varphi,v) := (\varphi,v,v,S_{\varphi}(v))
\end{equation*}
on $T\DRd{q}$ or of the extended Ebin vector field
\begin{equation*}
  X_{q,u_{0}}(\varphi) := A_{\varphi}^{-1}\left(\frac{1}{J_{\varphi}} (d\varphi^{-1})^{t} Au_{0} \right)
\end{equation*}
on $\DRd{q}$, for each $u_{0}\in \HRd{q}$.

\begin{theorem}\label{thm:smooth_flow-Hq}
  Let $A$ be an elliptic Fourier multiplier in the class $S^{r}$ with $r \ge 1$. Let $q > 1 + d/2$ with $r \le q$. Consider the geodesic flow on the tangent bundle $T\DRd{q}$ induced by the inertia operator $A$. Then, given any $(\varphi_{0},v_{0})\in
  T\DRd{q}$, there exists a unique non-extendable geodesic
  \begin{equation*}
    (\varphi, v)\in C^{\infty}(J_{q},T\DRd{q})
  \end{equation*}
  on the maximal interval of existence $J_{q}$, which is open and contains $0$.
\end{theorem}

To state a well-posedness result for the EPDiff equation in $\HRd{q}$, we need first to recall that we cannot conclude from theorem~\ref{thm:smooth_flow-Hq} that the curve
\begin{equation*}
  t \mapsto u(t) = v(t) \circ \varphi^{-1}(t), \qquad J_{q} \to \HRd{q}
\end{equation*}
is smooth, because the mapping
\begin{equation*}
  (v,\varphi) \mapsto v \circ \varphi^{-1}, \qquad \HRd{q} \times \DRd{q} \to \HRd{q}
\end{equation*}
is only continuous. However, using \cite[Theorem 1.1]{IKT2013} (and Remark 1.5 therein), we deduce\footnote{Strictly speaking, to apply theorem 1.1 in~\cite{IKT2013}, we should take $q > 2 + d/2$ in corollary~\ref{cor:Euler-well-posedness-Hq} because of the way theorem 1.1 is formulated but this is artificial. See ~\cite[Corollary B.6]{EK2014} for a proof in dimension 1.} that the curve
\begin{equation*}
  t \mapsto u(t) = v(t) \circ \varphi^{-1}(t), \qquad J_{q} \to \HRd{q-1}
\end{equation*}
is $C^{1}$ which leads to the following result on the initial value problem for the EPDiff equation.

\begin{corollary}\label{cor:Euler-well-posedness-Hq}
  Let $A$ be an elliptic Fourier multiplier in the class $S^{r}$ with $r \ge 1$. Let $q > 1 + d/2$ with $r \le q$. The corresponding Euler--Arnold equation~\eqref{eq:Diff-Euler-Arnold} has, for any initial data
  $u_{0}\in\HRd{q}$, a unique non-extendable solution
  \begin{equation*}
    u\in C^{0}(J_{q},\HRd{q}) \cap C^{1}(J_{q},\HRd{q-1}).
  \end{equation*}
  The maximal interval of existence $J_{q}$ is open and contains $0$.
\end{corollary}

\begin{remark}
  The same results hold for the geodesic flow on $T\DTd{q}$ with an elliptic inertia operator $A$ in the toroidal class $S^{r}$ with $r \ge 1$.
\end{remark}

It was pointed out in~\cite[Theorem 12.1]{EM1970} that the maximal interval of existence $J_{q}$ is, in fact, independent of the
parameter $q$, due to the invariance of the spray under the conjugate action of the translation group (or the rotation group in the case of the torus). This allows to avoid Nash--Moser type schemes to prove local existence of smooth geodesics in the Fr\'{e}chet category.

\begin{lemma}\label{lem:nlng}
  Given $(\varphi_{0},u_{0})\in T\DRd{q+1}$, we have
  \begin{equation*}
    J_{q+1}(\varphi_{0},u_{0})= J_{q}(\varphi_{0},u_{0}),
  \end{equation*}
  for $q> 1 + d/2$ and $r \le q$. The same result holds on the torus $\Td$.
\end{lemma}

\begin{proof}[Sketch of proof]
  We will do the proof for $\Rd$ (the proof is similar for $\Td$). Consider the standard action of $\Rd$ on itself
  \begin{equation*}
    (c,x) \mapsto t_{c}(x) = x + c, \qquad x,c \in \Rd\,.
  \end{equation*}
  Note that, even if the translation $t_{c}$ does not belong to $\DRd{q}$, the following conjugate action
  \begin{equation*}
    \psi : (c,\varphi) \mapsto t_{-c} \circ \varphi \circ t_{c}
  \end{equation*}
  of $\Rd$ on $\DRd{q}$ is well defined. Moreover, for each fixed $c \in \Rd$, the transformation $\psi_{c} := \psi(c, \cdot)$ is a smooth Riemannian isometry for the metric $G$ defined by \eqref{eq:definition-metric} on the Hilbert manifold $\DTd{q}$, when $A$ is a Fourier multiplier. Therefore, the geodesic spray $F_{q}$ is invariant under the induced action of $\psi$ on $T\DRd{q}$ and the same is true for its flow $\Phi_{q}$. Hence
  \begin{equation}\label{eq:flow-equivariance}
    \Phi_{q}(t,T\psi_{c}(\varphi_{0},u_{0})) = T\psi_{c} \left[\Phi_{q}(t,(\varphi_{0},u_{0}))\right],
  \end{equation}
  for all $t\in J_{q}(\varphi_{0},u_{0})$ and $c\in\Rd$. Observe now that, if $(\varphi_{0},u_{0})\in T\DRd{q+1}$, then, the mapping
  \begin{equation*}
    c \mapsto T\psi_{c} (\varphi_{0},u_{0}), \quad \Rd \to T\DRd{q}
  \end{equation*}
  is $C^{1}$. Moreover, if $(e_{i})$ denotes the canonical basis of $\Rd$, we have
  \begin{equation*}
    \left.\frac{d}{ds}\right|_{s=0} T\psi_{se_{i}}(\varphi_{0},u_{0})) = (\partial_{i}\varphi_{0}, \partial_{i} u_{0})\,.
  \end{equation*}
  Therefore, if $(\varphi_{0},u_{0})\in T\DRd{q+1}$, we get from~\eqref{eq:flow-equivariance}
  \begin{equation*}
    \partial_{(\varphi,v)}\Phi_{q}(t,(\varphi_{0},u_{0})).(\partial_{i}\varphi_{0}, \partial_{i} u_{0}) = (\partial_{i}\varphi(t), \partial_{i} v(t)).
  \end{equation*}
  But
  \begin{equation*}
    \partial_{(\varphi,v)}\Phi_{q}(t,(\varphi_{0},v_{0})).(\partial_{i}\varphi_{0}, \partial_{i} u_{0}) \in \HRd{q}\times\HRd{q},
  \end{equation*}
  for $1 \le i \le d$, and hence
  \begin{equation*}
    (\varphi(t),v(t))\in T\DRd{q+1}\quad\text{for all}\quad t\in J_{q}(\varphi_{0},v_{0}).
  \end{equation*}
  We conclude therefore that
  \begin{equation*}
    J_{q}(\varphi_{0},v_{0}) = J_{q+1}(\varphi_{0},v_{0}),
  \end{equation*}
  which completes the proof.
\end{proof}

\begin{remark}\label{rem:noloss_nogain}
  Lemma~\ref{lem:nlng} states that there is \emph{no loss} of spatial
  regularity during the evolution. By reversing the time direction, it
  follows from the unique solvability that there is also \emph{no gain} of
  regularity.
\end{remark}

\begin{remark}
  A similar result as lemma~\ref{lem:nlng} holds for the flow of Ebin's vector field $X_{q}$ on $\DRd{q}$. Indeed, we have
  \begin{equation*}
    X_{q, u_{0} \circ t_{c}}(\psi_{c} \cdot \varphi) = X_{q, u_{0}}(\varphi) \circ t_{c}\,,
  \end{equation*}
  which allows to establish a well-posedness result in the smooth category, using the particle-trajectory method, as well.
\end{remark}

We get therefore the following local existence result.

\begin{theorem}\label{thm:smooth_flow-Hinfty}
  Let $A$ be an elliptic Fourier multiplier in the class $S^{r}$ with $r \ge 1$ and
  consider the geodesic flow on the tangent bundle $T\DiffRd$. Then, given any initial data $(\varphi_{0},u_{0})\in
  T\DiffRd$, there exists a unique non-extendable geodesic
  \begin{equation*}
    (\varphi, v)\in C^\infty(J,T\DiffRd)
  \end{equation*}
  on the maximal interval of existence $J$, which is open and contains $0$. The same result hold for the torus.
\end{theorem}

We also obtain local well-posedness of the EPDiff equation.

\begin{corollary}\label{cor:Euler-well-posedness-Hinfty}
  The corresponding Euler equation has for any initial data
  $u_{0}\in\CS(\RR^{d})$ a unique non-extendable smooth solution
  \begin{equation*}
    u\in C^{\infty}(J,\CS(\RR^{d})).
  \end{equation*}
  The maximal interval of existence $J$ is open and contains $0$. The same result hold for the torus $\Td$.
\end{corollary}

\begin{corollary}
  The geodesic flow of the $H^{s}$ metric on $\DiffRd$, induced by the inertia operator $\Lambda^{2s} = \op{(1 + \abs{\xi}^{2})^{s}}$ is locally well-posed for $s \ge 1/2$. The same result holds for the torus $\Td$.
\end{corollary}

\begin{remark}
  It is a well-known fact that if $\norm{u_{x}(t)}_{\infty}$ is bounded on every bounded subinterval of $J$ then the same holds for all the $H^{q}$ norm of $v(t)$ (see~\cite{EK2014a} for a detailed proof, for instance) and hence that $J = \RR$. This observation allows to conclude that the geodesic flow of the $H^{s}$ metric is \emph{globally well-posed} for $s > 1 + d/2$ due to the Sobolev inequality $\norm{u_{x}}_{\infty} \lesssim \norm{u}_{H^{s}}$ and the fact that $\norm{u(t)}_{H^{s}}$ is constant along the flow.
\end{remark}


\end{document}